\documentclass[11pt]{article}
\usepackage{amsfonts}
\usepackage{graphicx}
\usepackage{amsmath}
\usepackage{amsthm}
\usepackage{amssymb}
\usepackage{amscd}
\usepackage{color}
\usepackage{url}

\usepackage{graphicx}
\usepackage{microtype}
\usepackage{enumerate}
\usepackage{pinlabel}
\usepackage[top=1.1in,bottom=1.6in,left=1.3in,right=1.3in]{geometry}

\theoremstyle{definition}

\newtheorem{thm}{Theorem}[section]
\newtheorem{cor}[thm]{Corollary}

\newtheorem{lem}[thm]{Lemma}
\newtheorem{defi}[thm]{Definition}
\newtheorem{rem}[thm]{Remark}
\newtheorem*{MW}{Milnor--Wood inequality \cite{Milnor, Wood}}

\renewcommand{\H}{\mathbb{H}}
\newcommand{\R}{\mathbb{R}}
\newcommand{\Z}{\mathbb{Z}}

\newcommand{\C}{\mathbb{C}}

\DeclareMathOperator{\id}{id}

\DeclareMathOperator{\SL}{SL}
\DeclareMathOperator{\SO}{SO}
\DeclareMathOperator{\GL}{GL}
\DeclareMathOperator{\Homeo}{Homeo}
\DeclareMathOperator{\Diff}{Diff}

\DeclareMathOperator{\Hom}{Hom}

\DeclareMathOperator{\isom}{Isom}

\title{Unboundedness of some higher Euler classes}
\author{Kathryn Mann}
\date{}
\begin{document}

\maketitle

\begin{abstract}
We study Euler classes in groups of homeomorphisms of Seifert fibered 3-manifolds.  In contrast to the familiar Euler class for $\Homeo_0(S^1)$ as a discrete group, we show that these Euler classes for $\Homeo_0(M^3)$ as a discrete group are \emph{unbounded} classes.   In fact, we give examples of flat topological $M$ bundles over a genus $3$ surface whose Euler class takes arbitrary values.
\end{abstract}

\setcounter{section}{1}

\bigskip

For a topological group $G$, let $G^\delta$ denote $G$ with the discrete topology, and $H^*(G; R) = H^*(BG^\delta; R)$ the group cohomology of $G$ with $R = \Z$ or $\R$ coefficients.   When $G$ is the group of homeomorphisms or diffeomorphisms of a manifold $M$, elements of $H^*(G; R)$ are characteristic classes of flat or foliated $M$ bundles with structure group $G^\delta$.    One says that a class is \emph{bounded} if it has a cocycle representative taking a bounded set of values on all $k$-chains of the form $(g_i, ..., g_k) \in G^k$.    Determining which classes are bounded is an interesting and often difficult question in its own right (see \cite{Monod} for an introduction to this and related problems in bounded cohomology) but particularly motivated in the case where $G$ is a subgroup of $\Homeo(M)$.  In this case, bounds on characteristic classes give obstructions for topological $M$-bundles to be flat.   On the flipside, showing that a class has \emph{no} bounded representative often amounts to constructing new examples of flat bundles.    
\medskip

The best known, and perhaps earliest example of a bounded class comes from Milnor \cite{Milnor}, who gave a bound on the Euler number of $\SL(2,\R)$ bundles over surfaces with discrete structure group.   
Wood \cite{Wood} generalized this argument to topological circle bundles ($\Homeo_0(S^1)$ naturally contains $\SL(2,\R)$ as a subgroup), to obtain a complete characterization of the oriented, topological circle bundles over surfaces that admit a foliation transverse to the fibers\footnote{In the smooth setting, this is equivalent to admitting a flat connection, hence, even in the topological case such bundles are called ``flat."  This is equivalent to the condition that the structure group reduces to a discrete group.}.  In modern language, their results can be reframed as follows:

\begin{MW}
The real Euler class in $H^2(\Homeo_0(S^1); \R)$ is bounded, and has (Gromov) norm equal to 1/2. 
\end{MW}

More generally, when $G$ is a real algebraic subgroup of $\GL(n,\R)$, it follows from \cite{Gromov} that the elements of $H^*(G; \R)$ obtained by the map $BG \to BG^\delta$ have bounded representatives, and explicit bounds on their norms have been computed in several cases.  See eg. \cite{BG, CO, DT} and references therein.  However, much less is known for large, nonlinear groups, in particular homeomorphism groups of manifolds.   

The first natural case to consider is that of any manifold $M$ such that $\pi_1(\Homeo_0(M))$ is either isomorphic to $\Z$ or has a $\Z$ summand.  (Here $\Homeo_0(M)$ denotes the identity component of $\Homeo(M)$.)  In this case $H^2(B\Homeo_0(M); \Z)$ has a $\Z$ summand, generated by an {\em Euler class} for topological $M$ bundles.  This pulls back to a {\em discrete Euler class} in $H^2(\Homeo_0(M); \Z)$, and we may ask which such classes are bounded.   The Milnor--Wood inequality is a positive answer to this question in the case $M = S^1$.  The only other known results are in dimension 2:  For $M=\R^2$, we also have that $\pi_1(\Homeo_0(\R^2)) = \Z$, and Calegari \cite{Calegari} showed that the discrete Euler class of topological $\R^2$ bundles is unbounded.  (In fact, he also showed unboundedness of its pullback to $H^2(\Diff_0(\R^2); \R)$.)    In the case of the 2-torus, $\pi_1(\Homeo_0(T^2)) = \Z \times \Z$, and an argument in \cite{MR} shows that discrete Euler classes are unbounded.

Here we address the same question for 3-manifolds.  
Following work of Hatcher, Ivanov,  McCullough and Soma, and Bamler--Kleiner \cite{Hatcher, Ivanov, MS, BK} on the generalized Smale conjecture, the inclusion $\isom_0(M) \to \Homeo_0(M)$ is known to be a homotopy equivalence on almost all geometric manifolds $M$ (the one open case is that were $M$ is non-Haken infranil).   In particular, this implies that for many closed, prime Seifert fibered 3-manifolds, rotation of the fibers gives either a homotopy equivalence $\SO(2) \to \Homeo_0(M)$, or at least a $\Z$ factor in $\pi_1(\Homeo_0(M))$, hence an Euler class for $M$ bundles.  
Our main result is that \emph{all} of these discrete Euler classes are unbounded.  Precisely, we show:

\begin{thm} \label{fibered thm}
Let $M$ be a closed Seifert fibered 3-manifold such that the inclusion $\SO(2) \hookrightarrow \Homeo_0(M)$ induces an inclusion of $\pi_1(\SO(2))$ as a direct factor in $\pi_1(\Homeo_0(M))$.  Then any class $\alpha \in H^2(\Homeo_0(M); \R)$ with nonzero image in $H^2(\SO(2); \R)$ is unbounded.  
\end{thm}

This is a direct consequence of the following stronger result.  

\begin{thm}  \label{surface group thm}
Let $M$ be as in  Theorem \ref{fibered thm}, and let $e \in H^2(\Homeo_0(M); \Z)$ have nonzero image in $H^2(\SO(2); \Z)$.  
Then, for any $k$ there exists a homomorphism $\rho$ from the fundamental group of a genus 3 surface $\Sigma$ to $\Homeo_0(M)$ such that 
$\langle \rho^*(e), [\Sigma] \rangle = k$. 
\end{thm}

Our proof is fundamentally different than Calegari's proof of unboundedness of the Euler class for $\Homeo_0(\R^2)$ bundles with discrete structure group, which uses non-compactness of $\R^2$ in an essential way.   It also differs considerably from the existing argument for unboundedness of cohomology classes in $\Homeo_0(T^2)$, which used the fact that $H^2(\Homeo_0(T^2); \Z) \cong \Z^2$ has a $\GL(2,\Z)$-action of the mapping class group of $T^2$.  

\medskip
Section \ref{prelim sec} contains some brief background on bounded cohomology, Gromov norm, and cohomology of homeomorphism groups, giving the tools to derive Theorem \ref{fibered thm} from Theorem \ref{surface group thm}.   The proof of Theorem \ref{surface group thm} is an explicit construction described in Section \ref{surface gp sec}.

\paragraph{Measure-preserving homeomorphisms.}   We contrast the result above with the measure-preserving case.  
Let $M$ be as in Theorem \ref{fibered thm}, and let $G$ be a subgroup of $\Homeo_0(M)$ that preserves a probability measure, or more generally a {\em content} on $M$.   In contrast to Theorem \ref{fibered thm}, work of Hirsch and Thurston implies that Euler classes pull back trivially to $H^2(G; \Z)$.    Their main theorem is the following: 

\begin{thm}[\cite{HT}]
Suppose $E \to B$ is a foliated bundle with structure group consisting of homeomorphisms that preserve a content on the fiber.  Then the induced map $H^*(B, \R) \to H^*(E, \R)$ is injective.
\end{thm}

\noindent To derive the vanishing result stated above, take any foliated $M$-bundle $p: E \to B$.  The pullback bundle $p^\ast(E) \to E$ has a section, so $p^\ast \rho^\ast(e) \in H^2(p^\ast E; \Z)$ is zero.  But if $E$ has content-preserving holonomy (i.e. its holonomy $\rho$ factors through a group $G$ as above), then Hirsch--Thurston implies that  $p^\ast$ is injective on cohomology, so $\rho^\ast(e) = 0$.  

Note that, by averaging any content over the $\SO(2)$ action on a Seifert-fibered manifold $M$, one may assume that it is invariant under rotation of fibers, and $\SO(2)$ includes in the group of content-preserving homeomorphisms.   This gives analogs of the 
Euler class in the group of content-preserving homeomorphisms as in Theorem \ref{fibered thm}; the remark above states that these are zero.   In particular, for the special case of the 2-dimensional torus, since $\pi_1(T^2) = \Z^2$ is amenable (so any action on a manifold $M$ has an invariant probability measure), this gives   

\begin{cor} \label{torus cor}
For $M$ as in Theorem \ref{fibered thm}, any $M$-bundle over $T^2$ with structure group $\Homeo_0(M)^\delta$ has zero Euler class.
\end{cor} 

\noindent Note that this statement would be {\em implied} by boundedness of $e \in H^2(\Homeo_0(M)^\delta; \Z)$.

\paragraph{Acknowledgements.}  The author thanks the referees for pointing out the references to Bamler-Kleiner and Hirsch-Thurston, including the argument given above.  Thanks also to Benson Farb, who first asked the author about analogs of Milnor-Wood (or its failure) in higher dimensions, and to Bena Tshishiku, Sam Nariman, and Wouter Van Limbeek for discussions and comments on this problem.   The author was partially supported by NSF grant DMS-1606254.

\section{Preliminaries}  \label{prelim sec} 

We quickly review the standard theory of bounded cohomology, as in Gromov \cite{Gromov}, and set up notation.  A reader who is well-acquainted with the subject can skip to Section \ref{h.e. sec} where we discuss cohomology of homeomorphism groups.   

For $M$ a manifold, and $a \in H_*(M; \R)$ an element of singular homology, there is a pseudonorm
$$\| a \| := \inf \{ \Sigma |c_i| :  [\Sigma \, c_i \sigma_i] = a \}$$ where the infimum is taken over all real singular chains representing $a$ in homology.  The $L_1$ norm on singular chains used in this definition gives a dual $L_\infty$ norm on singular cochains; and the set of bounded cochains forms a subcomplex of $C^*(M)$.  The cohomology of this complex is the \emph{bounded cohomology} $H^*_b(M; \R)$ of $M$.  The (pesudo-)\emph{norm}, $\| \alpha \|$, of a cohomology class $\alpha$ is the infimum of the $L_\infty$ norms of representative cocycles; and if $\| \alpha \|$ is finite we say that it is a \emph{bounded class}.

One can extend these definitions quite naturally to the Eilenberg--MacLane group cohomology.  Recall that, for a discrete group $G$, the set of \emph{inhomogeneous $k$-chains}, $C_k(G)$, is the free abelian group generated by $k$-tuples $(g_1, ..., g_k) \in G^k$ with an appropriate boundary operator.
The homology of this complex is the (integral) group homology $H_k(G; \Z)$; and $H_k(G; \R)$ is the homology of the complex $C_*(G) \otimes \R$.  The homology of the dual complexes $\Hom(C_k, \Z)$ and $\Hom(C_k, \R)$ give the \emph{group cohomology}
$H^k(G; \Z)$ and $H^k(G; \R)$ respectively.    As in the singular homology case above, there is a natural $L_1$ norm on $k$-chains given by $\| \sum s_i (g_{i,1}, ..., g_{i,k}) \| = \sum |s_i|$, which descends to a pseudonorm on homology by taking the infimum over representative cycles.  We also have a dual $L_\infty$ norm on $C^k(G)$, and for $\alpha \in H^*(G; \R)$ we define $$\| \alpha \| := \inf \{ \|c\|_\infty : [c] = \alpha \}.$$  Again, bounded (co)cycles are those with finite norm.    Note that $\| \alpha \|$ is finite if and only if there exists $D$ such that $|\alpha (g_1, g_2, ... , g_k)| < D$ holds for all $(g_1, g_2, ..., g_k) \in G^k$.  

A remarkable theorem of Gromov allows one to pass between groups and spaces: 
\begin{thm}[\cite{Gromov}]
There is a natural isometric isomorphism $H^*_b(\pi_1(M); \R) \to H^*_b(M; \R)$.
\end{thm}

\paragraph{Computing norms.}
In degree two, there is an effective means of estimating the norm of a cohomology class through representations of surface groups.  
For any space $X$, a class $c \in H_2(X; \Z)$ can always be represented as the image of a map from an orientable (possibly disconnected) surface $\Sigma$ into $X$.   If $X$ is a $K(G,1)$, then we may assume $\Sigma$ has no $S^2$ components.  
Supposing additionally that $\Sigma$ is  connected, such a map induces a homomorphism $\rho: \pi_1(\Sigma) \to G$.  Thus, on the level of group cohomology we have $c = \rho_*([\Sigma])$ and
$$\langle \alpha, c \rangle = \langle \rho^*(\alpha), [\Sigma] \rangle.$$
It is easy to verify that $[\Sigma]$ has norm $-2 \chi(\Sigma)$ (See \cite[\S 2]{Gromov} for the computation.)  Hence, we have $\| c\| \leq -2 \chi(\Sigma)$.  
 Thus, to show a cohomology class $\alpha$  is \emph{unbounded}, it suffices to 
 show that
$$\sup_{\rho: \pi_1(\Sigma) \to G} \frac{\langle \rho^*(\alpha), [\Sigma] \rangle}{2 \chi(\Sigma)} = \infty,$$ where the supremum is taken over all homomorphisms from surface groups into $G$.  

Although our goal here only requires us to show unboundedness of some classes, the above can actually be used to compute the norm of a class $\alpha$ in second bounded cohomology.   Matsumoto--Morita and Ivanov showed (independently) that, for any topological space $X$, Gromov's semi-norm on $H^2_b(X; \R)$ is in fact a {\em norm} \cite{MM, Ivanov2}.  Hence $H^2_b(X; \R)$ is a Banach space, with the quotient of $H_2(X;\R)$ by the zero-norm subspace as its dual; and in integral cohomology, the zero norm subspace is precisely the chains representable by maps of surfaces consisting of $S^2$ and $T^2$ components.

Returning to our situation, if $\Sigma$ is a connected surface of genus $g \geq 1$, the quantity $\langle \rho^*(\alpha), [\Sigma] \rangle$ of interest can be easily read off from a central extension.  
Recall that, for any abelian group $A$, there is correspondence between $H^2(G; A)$ and central extensions of $G$ by $A$.
If $\alpha \in H^2(G; \Z)$ is represented by the extension $0 \to \Z \to \hat{G} \to G \to 1$, then $\rho^*(\alpha)$ is represented by the pullback $0 \to \Z \to \rho^*(\hat{G}) \to \pi_1(\Sigma) \to 1$.   
The fundamental group of $\Sigma$ has a standard presentation 
$$\pi_1(\Sigma) = \langle a_1, b_1, ..., a_g, b_g \mid \prod_{i=1}^g [a_i, b_i] \rangle$$  
and the integer $\langle \rho^*(\alpha), [\Sigma] \rangle$ can be computed as follows.  Take 
 lifts $\widetilde{a}_i, \widetilde{b}_i$ of the generators $a_i$ and $b_i$ to elements of $\rho^*(\hat{G})$.  Since this is a central extension, the value of any commutator $[\widetilde{a}_i, \widetilde{b}_i]$ is independent of the choice of lifts $\widetilde{a}_i$ and $\widetilde{b}_i$.   The product of commutators $\prod_{i=1}^g [\widetilde{a}_i, \widetilde{b}_i]$ projects to the identity in $\pi_1(\Sigma)$, so can be identified with an element $n \in \Z$.  One checks easily from the definition that $n = \langle \rho^*(\alpha), [\Sigma] \rangle$.  
 
We note that, although not framed in the language of bounded cohomology, this strategy for computation is already present in Milnor and Wood's work in \cite{Milnor} and \cite{Wood} respectively.

\subsection{Euler classes of homeomorphism groups}  \label{h.e. sec} 

This section describes the known analogs of the Euler class in $\Homeo_0(M)$, for various manifolds $M$, explaining and justifying some of the remarks made in the introduction.  For simplicity, we always assume manifolds are closed.  

As mentioned in the introduction, whenever $M$ is a manifold with a circle action $\SO(2) \to \Homeo_0(M)$ such that the induced map on $\pi_1$ is inclusion of a direct factor, then  $B\Homeo_0(M)$ has a $B\SO(2) = \C P^\infty$ factor, giving an Euler class in second cohomology.  
While we are primarily concerned with the cohomology of discrete groups, a remarkable theorem of Thurston says that, in the very special case of homeomorphism groups of manifolds, this agrees with the cohomology of $B\Homeo$. 

\begin{thm}[Thurston \cite{Thurston}]
Let $M$ be a differentiable manifold.  Then the map $B\Homeo(M)^\delta \to B \Homeo(M)$ induced by the identity map 
$\Homeo(M)^\delta \to \Homeo(M)$ is an isomorphism on homology. 
\end{thm}

\noindent It follows from the theorem that the same statement holds for the identity components $\Homeo_0(M)^\delta \to \Homeo_0(M)$.  
Note that Thurston's theorem implies, in particular, the Euler class and its powers are the only characteristic classes of flat, oriented topological circle bundles.   

Unfortunately, there are not very many other manifolds where the homotopy type of (or at least the cohomology of) the identity component of their homeomorphism group is known.  In dimension 2, we know that $\Homeo_0(\Sigma)$ is contractible for any compact surface of negative Euler characteristic by \cite{EE}.  As mentioned in the introduction, $\SO(2) \to \Homeo_0(\R^2)$ is a homotopy equivalence, but unlike the $M = S^1$ case, the 
Euler class of $\Homeo_0(\R^2)^\delta$ is unbounded by \cite{Calegari}.  
For $M = T^2 = S^1 \times S^1$,  the inclusion $\SO(2) \times \SO(2) \to \Homeo_0(T^2)$ is a homotopy equivalence.   Thus $\Z^2 \cong H^2(B\Homeo(T^2);  \Z) \cong H^2(\Homeo(T^2);  \Z)$.   A direct computation, given in \cite[\S 4.2]{MR}, shows that both generators of $H^2(\Homeo(T^2);  \Z)$ are unbounded.  

The Seifert fibered 3-manifold case, of interest to us, provides essentially the only other examples where the homotopy type of $\Homeo_0(M)$ is both known and known to have a homotopically nontrivial $\SO(2)$ subgroup.    
For Haken manifolds, this is due to the following theorem of Hatcher and Ivanov.  

\begin{thm}[\cite{Hatcher}, \cite{Ivanov}]
Suppose $M$ is a closed, orientable, Haken, Seifert-fibered $3$-manifold.  Then the inclusion $S^1 \to \Homeo_0(M)$ by rotations of the fibers is a homotopy equivalence, except in the case $M=T^3$ where $\Homeo_0(T^3) \cong T^3$.
\end{thm}
  
We remark that Hatcher's original proof was in the PL category, but (as noted by Hatcher) this is equivalent to the topological category by the triangulation theorems of Bing and Moise \cite{Bing, Moise}.  Ivanov's proof of the theorem above is for groups of diffeomorphisms, but an argument due to Cerf, together with Hatcher's later proof of the Smale conjecture implies that the inclusion of $\Diff(M^3)$ into $\Homeo(M^3)$ is a homotopy equivalence; this makes the smooth category equivalent as well. 

McCullough--Soma \cite{MS} proved $\Homeo_0(M) \cong S^1$ for the small Seifert-fibered non-Haken manifolds with $\H^2 \times \R$ and $\widetilde{SL(2,\R)}$ geometries.    For spherical manifolds, Bamler-Kleiner's recent proof of the Smale conjecture \cite{BK} shows that the inclusion $\isom(M) \to \Homeo(M)$ is always a homotopy equivalence (and gives a new proof of contractibility of $\Homeo_0(M)$ when $M$ is hyperbolic).   This gives many examples of manifolds satisfying the condition of Theorem \ref{fibered thm}, including various families of lens spaces and several manifolds with non-cyclic fundamental group.   See \cite{HKMR} for a table of of homotopy types of isometry groups for spherical manifolds, as well as a good exposition on the problem and a proof (independent of Bamler--Kleiner) applicable in many specific cases.

\section{Proof of Theorems \ref{fibered thm} and \ref{surface group thm}} \label{surface gp sec}

Let $M$ be a Seifert fibered 3-manifold, and let $G = \Homeo_0(M)$.  Let $\iota:  \SO(2) \to G$ be the action of rotating the fibers, and suppose that $\iota$ induces an inclusion $\Z \cong \pi_1(\SO(2)) \to \pi_1(G)$ as a factor in a splitting as a direct product.    
Let $\widetilde{G}$ be the covering group of $G$ corresponding to the subgroup $\pi_1(G)/\iota(\Z) \subset \pi_1(G)$.   (Recall that $G$ is locally contractible by Cernavskii \cite{Cernavskii} or Edwards--Kirby \cite{EK}, so standard covering space theory applies here.)
If $\iota$ is also surjective on $\pi_1$, for instance, a homotopy equivalence, then $\widetilde{G}$ is the universal covering group of $G$.  In general, it is a central extension $0 \to \Z \to \widetilde{G} \to G \to 1$.

We will show that this central extension represents a class $e$ in $H^2(\Homeo_0(M)^\delta; \Z) \cong H^2(B\Homeo_0(M); \Z) \cong \Z$ that is \emph{unbounded}.  This will prove Theorem \ref{fibered thm}.   
Following the framework discussed in Section \ref{prelim sec}, to show that $e$ is unbounded, it suffices to construct representations of surface groups $\rho: \pi_1(\Sigma) \to \Homeo_0(M)$ with $\langle \rho^*(e), [\Sigma] \rangle/\chi(\Sigma)$ arbitrarily large.  Although, in using this strategy, \emph{a priori}  one may need to vary the genus of surface to construct representations with increasingly large values of $\langle \rho^*(e), [\Sigma] \rangle/\chi(\Sigma)$, in this case we need only to work with a surface of genus 3.  

Put otherwise, we will show how to construct commutators $[a_i, b_i]$ with $a_i$ and $b_i \in G$ (for $i = 1, 2, 3$), such that $\prod_{i=1}^3 [a_i, b_i] = \id$, but where the product of lifts $\prod_{i=1}^3 [\widetilde{a}_i, \widetilde{b}_i]$ to $\widetilde{G}$ represent unbounded covering transformations. This will prove Theorem \ref{surface group thm}.  

\medskip

The first step is a local construction of bump functions.    

\begin{defi} \label{bump def}
A \emph{standard bump function} on $D^2$ is a function $D^2 \to \R$, which, after conjugation by some $h \in \Homeo_0(D^2)$ agrees with 
$$f(re^{i \theta}) = 
\left\{ \begin{array}{rr} 1 &\text{ if } r<1/3 \\
2-3r &\text{ if } 1/3 \leq r \leq 2/3 \\
0 &\text{ if } r > 2/3 
\end{array} \right.
$$
\end{defi}

What we have in mind as particular examples are piecewise linear (or piecewise smooth) functions $f: D^2 \cong [-1, 1] \times [-1,1] \to \R$ that are identically 0 on a neighborhood of the boundary, identically 1 on a neighborhood of $(0,0)$, and with the level sets $f^{-1}(p)$  for $p \in (0,1)$ given by piecewise linear (or piecewise smooth) curves.  Moreover, these should have the property that some line $\lambda$ from $0$ to $\partial([-1, 1] \times [-1,1])$ is transverse to each level set of $f$, with $f$ monotone along $\lambda$.    
In this case, one can easily construct the conjugacy $h$ to the function above defined on the round disc as follows.  For $p \in (0,1)$, let $\ell_p$ be the total arc length of $f^{-1}(p)$ and, for $x \in f^{-1}(p)$, let $\ell_p(x)$ denote the arc length of the segment of $f^{-1}(p)$ (oriented as the boundary of $f^{-1}([p, 1])$) from $\lambda \cap f^{-1}(p)$ to $x$.  
Then, for $x \in f^{-1}(p)$, set $h(x) = \frac{2-p}{3} e^{i \ell_p(x)/\ell(p)}$.   One may then extend $h$ arbitrarily to a homeomorphism defined on $f^{-1}(0)$ and $f^{-1}(1)$.

\begin{lem}  \label{disc lem}
Let $T = D^2 \times S^1$ be a $(p, q)$ standard fibered torus, let $f$ be a standard bump function, and let $k \in \R$.  
There exist $a, b \in \Diff(T)$ such that the commutator $b^{-1}a^{-1}ba$ preserves fibers and rotates the fiber $\{x\} \times S^1$ by $2\pi k f(x)$ if $x \neq 0$, and the exceptional fiber by $2\pi q k$.

\end{lem} 

\begin{proof}
We take local coordinates to identify $D^2$ with the rectangle $[-3, 3] \times [-3, 3] \subset \R^2$, so that the exceptional fiber passes through $(0,0)$, and we work in the PL setting.  First, define $\phi$ to be a standard bump function that is identically $1$ on $[-1, 1]^2$, zero on the complement of $[-2, 2]^2$, and in the topological annulus between these regions of definition, it is linear on each of the four sets cut out by the diagonals of $[-3, 3]^2$.   Level sets of $a$ are shown in Figure \ref{bump fig} left.   
For a point $(x, s)$ in $[-3, 3]^2 \times S^1$, define $a(x,s) =(x, s + 2 \pi q k \phi(x))$ if $x \neq (0,0)$ (i.e. a rotation of the fiber over $x$ by $2 \pi q k \phi(x)$), and define $a$ to be a rotation by $2 \pi k$ on the exceptional fiber.  

To construct $b$, first define $F: [-3, 3] \to [-3, 3]$ by 
$$F(u)  = 
\left\{ \begin{array}{rr} u &\text{ if } u \geq1 \\
\tfrac{u + 2}{3} &\text{ if } -2 < u < 1 \\
3(u+2) &\text{ if } -3 \leq u \leq -2 
\end{array} \right.
$$
and define $b$ on $[-3, 3] \times [-3,3] \times S^1$ by $b(u, v, e^{i \theta}) = (F(u), v, e^{i \theta})$.

Since both $a$ and $b$ preserve fibers, $ba^{-1}b^{-1}$ does as well.  Moreover, $ba^{-1}b^{-1}$ rotates the fiber through a point $x \in [-3,3]^2$ by $- 2\pi q k \phi(b^{-1}(x))$ for $x \neq 0$, and by $- 2\pi k \phi(b^{-1}(x))$ on the exceptional fiber.  Composing $a \circ ba^{-1}b^{-1}$ gives a function which rotates a non-exceptional fiber over a point $x$ by $2\pi q k(\phi(x) - \phi(b^{-1}(x)))$, this gives a standard bump function whose level sets are depicted in Figure \ref{bump fig} right; it is the result of adding the bump functions of the other figures.   
\end{proof} 

  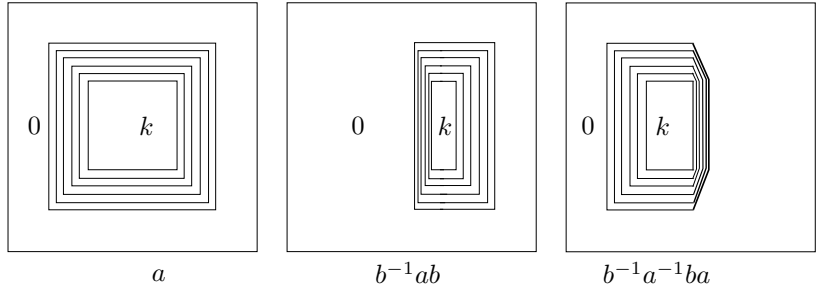
\begin{figure*}
   \labellist 
  \small\hair 2pt
   \pinlabel $a$ at 150 10
   \pinlabel $ba^{-1}b^{-1}$ at 350 12
   \pinlabel $aba^{-1}b^{-1}$ at 550 12
   \pinlabel 0 at 50 130 
   \pinlabel $k$ at 140 130 
   \pinlabel 0 at 310 130 
   \pinlabel -$k$ at 380 130 
   \pinlabel 0 at 495 130 
 \pinlabel $k$ at 555 130 
   \endlabellist
     \centerline{ \mbox{
 \includegraphics[width = 5.5in]{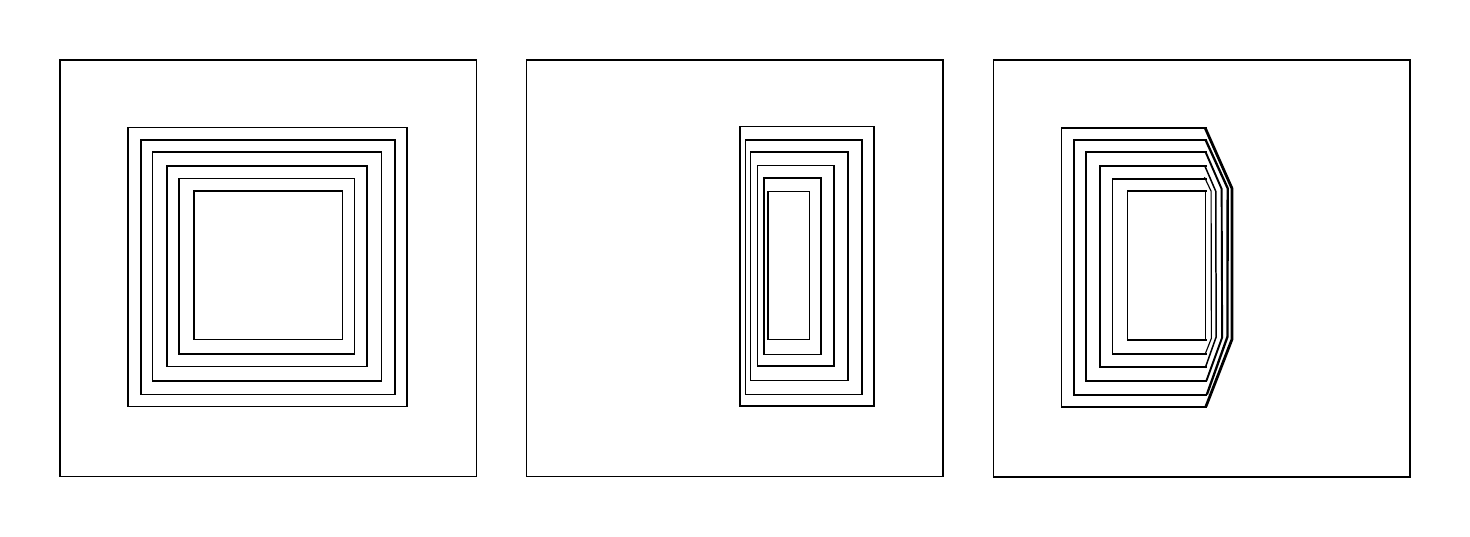}}}
 \caption{Level sets of PL bump functions}
  \label{bump fig}
  \end{figure*}

The next step is to glue the bump functions given by Lemma \ref{disc lem} together into a nice partition of unity, subordinate to an open cover consisting of only three sets.  

\begin{lem} \label{good partition} 
Let $S$ be an orientable topological surface.  There exists an open cover $\mathcal O = \{O_1, O_2, O_3\}$ of $S$, with each $O_i$ a union of disjoint homeomorphic open balls, and a partition of unity $\lambda_i$ subordinate to $\mathcal O$ such that the restriction of $\lambda_i$ to any connected component of $O_i$ is a standard bump function. 
\end{lem}

\begin{proof}  
Let $\Gamma = (V, E)$ be a degree three graph on $S$, with polygonal faces.   For example, $\Gamma$ may be constructed as the dual graph to a triangulation of $S$.  
First we define the sets in the cover $\mathcal{O} = \{ O_1, O_2, O_3 \}$.   Let $N_\delta$ denote the union of the $\delta$-neighborhoods of the edges in $\Gamma$.  Fixing an appropriate metric and PL structure on $S$, we may assume that the boundary of $N_\delta$, for any sufficiently small $\delta>0$, consists of line segments parallel to the edges of $\Gamma$.  

Fixing $\delta$,  let $O_1 = S \setminus N_{\delta/2}$.  Choose $\delta$ small enough so that connected components of $O_1$ are in one to one correspondence with faces of the graph, each the complement of a small $\delta/2$-neighborhood of the boundary of the face.   For each edge $e$, let $m_e$ denote its midpoint.  In a neighborhood of $m_e$, $N_\delta$ has natural local coordinates as $(-\delta, \delta) \times (-1, 1)$  with the edge given by $0 \times (-1,1)$, $m_e = (0, 0)$ and lines $\{p\} \times (-1,1)$ parallel to the edge.  We assume that $\delta$ is small enough so that we may choose these neighborhoods of midpoints to be pairwise disjoint and let $U_e$ denote the neighborhood containing $m_e$. Let $O_2$ be the union $\bigcup_{e \text{ edge}} U_e$.   Finally, let $X$ be the union of the sub-neighborhoods $(-\delta, \delta) \times [-1/2, 1/2]$ and let $O_3$ be the complement of $X$ in $N_\delta$.    See figure \ref{cover fig} for a local picture. 

  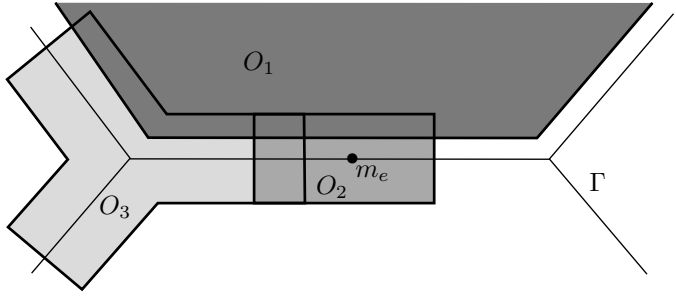
\begin{figure*}
   \labellist 
  \small\hair 2pt
   \pinlabel $O_1$ at 150 135
   \pinlabel $O_2$ at 193 62
   \pinlabel $O_3$ at 65 50
   \pinlabel $\Gamma$ at 350 65 
   \pinlabel $\bullet$  at 205 78
   \pinlabel $m_e$ at 217 70
   \endlabellist
     \centerline{ \mbox{
 \includegraphics[width = 3.5in]{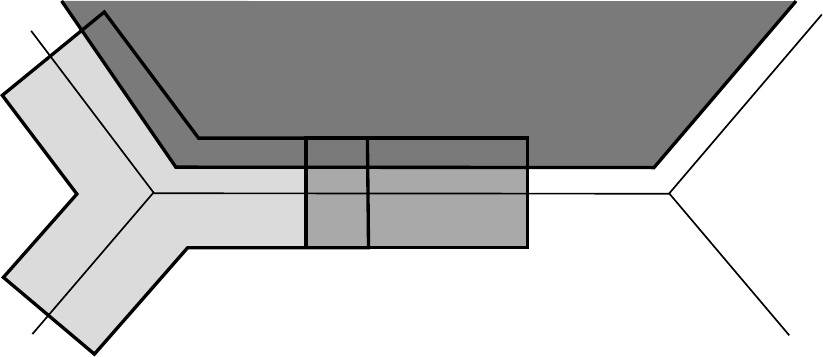}}}
 \caption{A cover supporting a good partition}
  \label{cover fig}
  \end{figure*}

We now construct the desired partition of unity, with $\lambda_i$ supported on $O_i$.   
Define $\lambda_1$ to be constant $1$ on $S \setminus N_{\delta}$, constant $0$ on $N_{\delta/2}$, and piecewise linear in the intermediate regions, with level sets consisting of polygons with edges parallel to the edges of $\Gamma$.  

Let $g = 1 - \lambda_1$, this is a function supported on $O_2 \cup O_3$.   Define $\lambda_2$ to agree with $g$ on the complement of $\bigcup_{e \in E} U_e$.   In the coordinates $U_e = (-\delta, \delta) \times (-1, 1)$ given above, define the restriction of $\lambda_2$ to $U_e$ to agree with $g$ on $(-\delta, \delta) \times (-1/2, 1/2)$, to be given by 
$\lambda_2(x, y) = 2(1-|y|)g(x,y)$ on $(-\delta, \delta) \times (-1, -1/2) \cup (-\delta, \delta) \times (1/2, 1)$, and then extend $\lambda_2$ to be $0$ elsewhere.  This gives a continuous (in fact, piecewise linear) bump function supported on $O_2$.  Finally, let $\lambda_3 = 1 - \lambda_1 - \lambda_2$, which is supported on $O_3$.  It is easily verified that this is a standard bump function, as in the example discussed after Definition \ref{bump def}.  
\end{proof}

To finish the proof of Theorem \ref{surface group thm}, let $M$ be a Seifert fibered $3$-manifold, and let $S$ be the base orbifold.  
Take a cover $\mathcal{O} = \{ O_1, O_2, O_3 \}$ of $S$ as given by Lemma \ref{good partition}.   Using the construction from Lemma \ref{good partition} starting with a graph on $S$, we may arrange for each exceptional fiber to be contained in only one set in $\mathcal O$, and also to have each connected component of each element of $\mathcal O$ contain at most one exceptional fiber.   Let $\{ \lambda_i\}$ be the partition of unity subordinate to this cover consisting of standard bump functions.   

Fix a connected component $B$ of some set $O_i \in \mathcal{O}$, and let $B \times S^1$ be the union of fibers over $B$.  By construction this is a $(p,q)$ standard fibered torus for some $p,q$.   
Fix $K \in \Z$.  Lemma \ref{disc lem} constructs homeomorphisms $a_B, b_B \in \Homeo_0(M^3)$ supported on $B \times S^1$ such that the commutator $[a_B, b_B]$ rotates each (nonexcptional) fiber over $\{x\} \times S^1$ by $2\pi K \lambda_i(x)$.   There is a natural path $a_B(t)$ from the identity in $\Homeo_0(M)$ to $a_B(1) = a_B$ by applying the construction of Lemma \ref{disc lem} to give rotations of a (non-exceptional) fiber through $x$ by $2 \pi q tK \lambda_i(x)$ at time $t$.   

Then $[a_B(t), b_B]$ gives a path from identity to $[a_B, b_B]$ that rotates (non-exceptional) fibers by $2 \pi q tK \lambda_i(x)$ at time $t$.  
Moreover, if $b_B(t)$ is any path from $b_B$ to the identity supported on $B$, then $[a_B(t), b_B]$ is homotopic rel endpoints to $[a_B(t), b_B(t)]$.  
 Let 
 $$a_i =  \prod_{B} a_B, \, \text{ and } \, a_i(t) =  \prod_{B} a_B(t)$$ 
 where the product is taken over all connected components of $O_i$.  Similarly, let 
 $$b_i = \prod_B b_B, \, \text{ and } \, b_i(t) =  \prod_{B} b_B(t).$$

Let $\widetilde{G}$ be the covering group of $G = \Homeo_0(M)$ as given at the beginning of this section; i.e. the central extension $0 \to \Z  \to  \widetilde{G} \to G \to 1$.   One  definition of this covering group is as the set of equivalence classes of paths based at the identity in $G$, where two paths are equivalent if they have the same endpoint and their union is an element of $\pi_1(G)$ that belongs to the subgroup $\pi_1(G)/\iota(\Z)$.   The group operation is pointwise multiplication, or equivalently, concatenation.  In this interpretation, the inclusion of $n \in \Z$ into $\widetilde{G}$ is given by a path $g_t$ in $G$, $t\in [0,1]$  that rotates (nonexceptional) fibers by an angle of $2\pi n t$ at time $t$.

Now we return to the machinery of Section \ref{prelim sec}.  Consider the map of a genus 3 surface group into $G$ where the images of the standard generators are $a_i$ and $b_i$ as defined above.    The paths $a_i(t)$ and $b_i(t)$ give lifts of $a_i$ and $b_i$ to $\widetilde{G}$, with commutator 
$[a_i(t), b_i(t)]$ a path from the identity to a map that rotates fibers by $2\pi K \lambda_i(x)$.   Hence, $\prod_{i=1}^3 [a_i(t), b_i(t)]$ represents $K \in \Z$.   Thus, if $\rho$ is the associated map of the surface group, and $e$ the Euler class in $H^2(G, \Z)$, this means that 
$\langle \rho_*(\Sigma), e \rangle = K$.  Since $K$ can be chosen arbitrarily, this proves Theorem \ref{surface group thm}.  
\qed

\begin{rem} 
The constructions above can likely be realized in the smooth category (i.e. with a homomorphism $\pi_1(\Sigma_3) \to \Diff_0(M)$), however, some more care is needed in the construction of the bump functions, as not all convex, smooth bump functions on a disc are smoothly conjugate.  
\end{rem}


\end{document}